\documentclass{amsart}

\usepackage{xypic}

\newtheorem{theorem}{Theorem}[section]
\newtheorem{lemma}[theorem]{Lemma}
\newtheorem{proposition}[theorem]{Proposition}
\newtheorem{corollary}[theorem]{Corollary}

\theoremstyle{definition} 
\newtheorem{definition}[theorem]{Definition}

\theoremstyle{remark} 

\numberwithin{equation}{section}

\newcommand{\Z}{\mathbb{Z}}
\newcommand{\OO}{\mathcal{O}} 

\newcommand{\F}{\mathfrak{F}}

\newcommand{\PP}{\mathbb{P}} 
\newcommand{\C}{\mathbb{C}}
\newcommand{\R}{\mathbb{R}} 
\newcommand{\T}{\mathbb{T}}
\newcommand{\A}{\mathbb{A}} 
\newcommand{\wA}{\widehat{\A}} 
\newcommand{\mA}{\mathcal{A}}
\newcommand{\mT}{\mathcal{T}}

\newcommand{\Iso}{\mathrm{Iso}}
\newcommand{\PC}{\PP^n(\C)}

\newcommand{\SU}{\mathrm{SU}}
\newcommand{\PSU}{\mathrm{PSU}}

\begin{document}

\title[Commutative Toeplitz operators on $\PC$]{Commutative
  $C^*$-algebras of Toeplitz operators on complex projective spaces}

\author[R.~Quiroga-Barranco]{Raul Quiroga-Barranco} 
\address{%
 Centro de Investigaci\'on en Matem\'aticas \\
 Apartado Postal 402 \\
 36000, Guanajuato, Gto. \\
 M\'exico.  
} 
\email{quiroga@cimat.mx} 

\thanks{The first named author was partially supported by SNI-Mexico
  and by the Conacyt grant no. 82979. The second named author was
  partially supported by a Conacyt postdoctoral fellowship. }

\author[A.~Sanchez-Nungaray]{Armando Sanchez-Nungaray}
\address{%
Centro de Investigaci\'on en Matem\'aticas \\
Apartado Postal 402 \\
36000, Guanajuato, Gto. \\
M\'exico.
}
\email{armandos@cimat.mx}



\begin{abstract}
  We prove the existence of commutative $C^*$-algebras of Toeplitz
  operators on every weighted Bergman space over the complex
  projective space $\PC$. The symbols that define our algebras are
  those that depend only on the radial part of the homogeneous
  coordinates. The algebras presented have an associated pair of
  Lagrangian foliations with distinguished geometric properties and
  are closely related to the geometry of $\PC$.
\end{abstract}

\maketitle

\section{Introduction}
The existence of nontrivial commutative $C^*$-algebras of Toeplitz
operators on bounded domains has shown to be a remarkably interesting
phenomenon (see \cite{GQV-disk}, \cite{QV-Ball1}, \cite{QV-Ball2},
\cite{QV-Reinhardt}, \cite{Nikolai-book}). Large families of symbols
defining such commutative algebras on every weighted Bergman space
have been proved to exist on the unit ball and Reinhardt domains.
Also, in \cite{Ernesto} it was proved the existence of commutative
$C^*$-algebras of Toeplitz operators for the sphere. All of the
examples exhibited in these references come equiped with a
distinguished geometry and can be described in terms of the isometry
group of the domain.

In this work we study the complex projective space $\PC$ and its
hyperplane line bundle $H$. It is well known that these objects
provide the usual setup to consider quantization through the use of
weighted Bergman spaces (see \cite{Schli}). Note that, due to the
compactness of $\PC$, the weight of the Bergman spaces has a discrete
set of values. Within this setup for $\PC$, we prove the existence of
commutative $C^*$-algebras of Toeplitz operators on all weighted
Bergman spaces. As in the case of previous works, our algebras can be
described in terms of the symmetries of $\PC$. To elaborate on this,
recall that for every hyperplane $P \subset \PC$, the open set $\PC
\setminus P$ can be canonically identified with the complex plane
$\C^n$. With respect to such identification, we prove that the symbols
that depend only on the radial components of the coordinates of $\C^n$
define commutative $C^*$-algebras of Toeplitz operators on every
weighted Bergman space over $\PC$ (see
Theorem~\ref{theo:comm-C*-algebra}); this sort of symbols are called
separately radial. Furthermore, we construct unitary equivalences of
the weighted Bergman spaces with suitable (finite dimensional) Hilbert
spaces such that the Toeplitz operators with separately radial symbols
are simultaneously turned into multiplication operators.

We also prove that our commutative $C^*$-algebras have a distinguished
geometry that describes them. Recall that the connected component of
the group of isometries of $\PC$ is given by the group $\SU(n+1)$ of
$(n+1)\times(n+1)$ unitary matrices with determinant $1$. It is proved
that the separately radial symbols are those invariant under the
action of the diagonal matrices of $\SU(n+1)$ (see
Lemma~\ref{lem:radialgroup}); in particular, this provides a
description of the separately radial symbols that depends on $\PC$
only, and not on a set of coordinates. Also, it allows us to show that
every commutative $C^*$-algebra obtained from separately radial
symbols has an associated pair of Lagrangian foliations with special
geometric properties (see
Theorem~\ref{theo:lagrframes-radialsymbols}). Such properties come
from the group $\A(n)$ of isometries of $\PC$ defined by the diagonal
matrices in $\SU(n+1)$. It turns out that $\A(n)$ is a maximal Abelian
subgroup of the group of isometries of $\PC$ (see
Section~\ref{sec:abel-lagr}). We prove that, up to an isometry,
$\A(n)$ is the only such maximal Abelian subgroup. This behavior is in
contrast with the existence of $n+2$ maximal Abelian subgroups of the
group of isometries of $\mathbb{B}^n$, which provides $n+2$
nonequivalent commutative $C^*$-algebras of Toeplitz operators (see
\cite{QV-Ball1} and \cite{QV-Ball2}). This is why for the projective
space $\PC$ there exist, up to an isometry, only one commutative
$C^*$-algebra of Toeplitz operators that has an associated pair of
foliations as described in
Theorem~\ref{theo:lagrframes-radialsymbols}. 

As for the contents of this work, in Section~\ref{sec:geom-PC} we
present some differential geometric preliminaries on
$\PC$. Section~\ref{sec:quantum} recalls the basic material on
quantization on compact K\"ahler
manifolds. Section~\ref{sec:quantum-PC} recollects the known results
specific to the quantization for complex projective spaces including
weighted Bergman spaces. Sections~\ref{sec:radial} and
\ref{sec:abel-lagr} contain our main results and proofs on the
existence of commutative $C^*$-algebras and their geometric
description; our techniques are closely related to those found in
\cite{QV-Ball1} and \cite{QV-Ball2}.

\section{Geometric preliminaries on $\PC$}\label{sec:geom-PC}
Recall that $\PC$ is the complex $n$-dimensional manifold that
consists of the elements $[w] = \C w\setminus\{0\}$, where $w \in
\C^{n+1}\setminus \{0\}$. For every $j = 0, \dots, n$, we have an open
set
$$
U_j = \{ [w] \in \PC : w_j \not= 0 \}
$$
and a biholomorphism $\varphi_j : U_j \rightarrow \C^n$ given by
$$
\varphi_j([w]) = \frac{1}{w_j} (w_0, \dots, \widehat{w}_j, \dots, w_n)
= (z_1, \dots, z_n),
$$
where the numbers $z_k$ are known as the homogeneous coordinates with
respect to the map $\varphi_j$. The collection of all such maps yields
the holomorphic atlas of $\PC$ and provides (equivalent) realizations
of $\C^n$ as an open dense conull subset of $\PC$. 

We refer to Example~6.3 of \cite{KNII} for the details on the
following construction of the Fubini-Study metric on $\PC$.

For every $j = 0, \dots, n$ consider the function $f_j : U_j
\rightarrow \C$ given by
\begin{equation}\label{potencial}
f_j([w]) = 
\sum_{k=0}^n \frac{w_k\overline{w}_k}{w_j\overline{w}_j} =
1 + \sum_{k=1}^n z_k \overline{z}_k, 
\end{equation}
for the above homogeneous coordinates $z_k$ with respect to
$\varphi_j$. Then, it is easily seen that
$$
\partial \overline{\partial} \log f_j = \partial \overline{\partial} \log f_k
$$
on $U_j \cap U_k$ for every $j,k= 0, \dots, n$. In particular, there
is a well defined closed $(1,1)$-form $\omega$ on $\PC$ given by
\begin{equation}\label{simplectic-form}
\omega = i \partial \overline{\partial} \log f_j,
\end{equation}
on $U_j$. This yields the canonical K\"ahler structure on $\PC$ for
which it is the Hermitian symmetric space with constant positive
holomorphic sectional curvature. The corresponding Riemannian metric
is known as the Fubini-Study metric. We refer to \cite{KNII} for these
and the rest of the remarks in this section on the geometry of $\PC$
induced by $\omega$.

With respect to the chart $\varphi_0$, we have the following induced
K\"ahler form on $\C^n$
$$
\omega_0 = (\varphi_0^{-1})^*(\omega) = 
i\frac{(1+|z|^2) \sum_{k=1}^n dz_k\wedge d\overline{z}_k -
  \sum_{k,l=1}^n \overline{z}_k z_l dz_k\wedge
  d\overline{z}_l}{(1+|z|^2)^2}.
$$
The volume element on $\PC$ with respect to the Fubini-Study metric
is defined by
$$
\Omega = \frac{1}{(2\pi)^n}\omega^n.
$$
The following result is a consequence of the definition of the volume
element of a Riemannian metric and its properties for Hermitian
symmetric spaces (see \cite{Helgason}).

\begin{lemma}
  The volume element on $\C^n$ induced by the Fubini-Study metric of
  $\PC$ is given by
  $$
  \Omega = \frac{1}{(2\pi)^n}\omega_0^n = \frac{1}{\pi^n}
  \frac{dV(z)}{(1 + |z_1|^2 + \dots + |z_n|^2)^{n+1}}
  $$
  where $dV(z) = dx_1\wedge dy_1 \wedge \dots \wedge dx_n\wedge dy_n$
  is the Lebesgue measure on $\C^n$.
\end{lemma}

We also recall that the Hopf fibration of $\PC$ is given by
\begin{align*} 
  \pi : S^{2n+1} & \rightarrow \PC \\ 
  w &\mapsto [w],
\end{align*} 
where $S^{2n+1} \subset \C^{n+1}$ is the unit sphere centered at the
origin.

We denote with $\SU(n+1)$ the Lie group of $(n+1)\times (n+1)$ unitary
matrices with determinant $1$. In other words, for $A$ a complex
$(n+1)\times (n+1)$ matrix with $\det(A) = 1$, we have $A \in
\SU(n+1)$ if and only if $A^*A = I_{n+1}$. For $Z_{n+1}$ the group of
$(n+1)$-th roots of unity in $\C$, we also consider the quotient Lie
group $\PSU(n+1,\C) = \SU(n+1,\C)/Z_{n+1}I_{n+1}$, and we denote with
$\lambda$ the natural quotient map of $\SU(n+1)$ onto $\PSU(n+1)$. It is
well known that these Lie groups are connected and compact.  The
following result is a consequence of Example~10.5 of \cite{KNII} and
Section~4 in Chapter~VIII of \cite{Helgason}.

\begin{proposition}\label{prop-Iso-PSU}
  The natural action of $\SU(n+1)$ on $S^{2n+1}$ induces a holomorphic
  action of $\PSU(n+1)$ on $\PC$ that satisfies the following
  properties:
  \begin{itemize}
  \item $\lambda(A)[w] = [Aw]$ for every $w \in S^{2n+1}$ and $A \in
    \SU(n+1)$.
  \item The induced $\PSU(n+1)$-action on $\PC$ realizes the connected
    component $\Iso_0(\PC)$ of the group of isometries of $\PC$ for
    the Fubini-Study metric.
  \end{itemize}
  Hence, we have an isomorphism of Lie groups $\Iso_0(\PC) \simeq
  \PSU(n+1)$.
\end{proposition}

\section{Quantum line bundles and quantization on compact K\"ahler
  manifolds}\label{sec:quantum}
Let $M$ be a K\"ahler manifold with K\"ahler form $\omega$. We will
now recall the notion of a quantum line bundle over $M$ which allows
to consider the Berezin-Toeplitz quantization for suitable sections
over $M$. We refer to \cite{Griffiths-Harris} and \cite{Schli} for
further details.

Suppose that $L \rightarrow M$ is a holomorphic line bundle. For any
such line bundle, we will denote with $\Gamma(U,L)$ the space of
smooth sections of $L$ over an open set $U$ of $M$. A Hermitian metric
$h$ on $L$ is a smooth choice of Hermitian inner products on the
fibers of $L$. For such metric, the pair $(L,h)$ is called a Hermitian
line bundle.

A connection $D$ on $L$ is given by assignments
$$
D : \Gamma(U,L) \rightarrow \Gamma^1(U,L),
$$
where $U$ is an open subset of $M$ and $\Gamma^k(U,L)$ denotes the
space of $L$-valued $k$-forms over $U$; also, $D$ must be complex
linear and satisfy the following property
$$
D(f\zeta) = df\otimes\zeta + fD\zeta,
$$
for every $f\in C^\infty(U)$ and $\zeta \in \Gamma(U,L)$. We can
extend $D$ as a derivation on $L$-valued forms so that it defines maps
$D : \Gamma^k(U,L) \rightarrow \Gamma^{k+1}(U,L)$. Then, the curvature
of $D$ is defined as $D^2 : \Gamma(U,L) \rightarrow \Gamma^2(U,L)$. It
is well known that $D^2$ is linear with respect to the multiplication
by smooth functions, which implies that $D^2$ defines a $L\otimes
L^*$-valued $2$-form. The latter is called the curvature of $D$.

Also, the connection $D$ is said to be compatible with $h$ if the
following conditions are satisfied:
\begin{itemize}
\item For every $\zeta_1,\zeta_2 \in \Gamma(U,L)$ and $X$ a smooth
  vector field over $U$ we have
  $$
  X(h(\zeta_1,\zeta_2)) = h(D_X \zeta_1,\zeta_2) +
  h(\zeta_1,D_X\zeta_2). 
  $$
\item For every holomorphic section $\zeta \in \Gamma(U,L)$ and $X$ a
  smooth vector field over $U$ of type $(0,1)$ we have
  $$
  D_X \zeta = 0.
  $$
\end{itemize}

The following result establishes the existence and uniqueness of
compatible connections. It also provides an expression for the
curvature form. We refer to \cite{Griffiths-Harris} for its proof.

\begin{proposition}\label{prop:conn_curv}
  For a Hermitian line bundle $(L,h)$ over a K\"ahler manifold $M$,
  there exist a unique connection $D$ that is compatible with $h$. The
  curvature can be considered as a complex-valued $(1,1)$-form
  $\Theta$. Moreover, if $\zeta$ is a local holomorphic section of $L$
  on the open set $U$, then on this open set the curvature is given by
  $\Theta = \overline{\partial}\partial\log(h(\zeta,\zeta))$.
\end{proposition}

The connection from Proposition~\ref{prop:conn_curv} is called the
Hermitian connection of the Hermitian line bundle.

A Hermitian line bundle $(L,h)$ over a K\"ahler manifold $M$ is called
a quantum line bundle if it satisfies the condition
$$
\Theta = -i\omega,
$$
where $\Theta$ is the curvature of $L$ and $\omega$ is the K\"ahler
form of $M$. By Proposition~\ref{prop:conn_curv}, this is equivalent
to requiring
$$
\omega = i\overline{\partial}\partial\log(h(\zeta,\zeta)),
$$
for every open set $U\subset M$ and every local holomorphic section of
$L$ defined on $U$.

For any Hermitian line bundle $(L,h)$ and $m \in \Z_+$, we denote
$L^m = L \otimes \dots \otimes L$ ($m$ times), which is itself a
Hermitian line bundle with the induced metric. We will denote the
latter by $h^{(m)}$. Furthermore, it is easy to see that the Hermitian
connection of $L^m$ is the induced connection of $D$ to the tensor
product. We will denote such connection by $D^{(m)}$.

We recall that the volume form of the K\"ahler manifold $(M,\omega)$ is
given by $\Omega = \omega^n/(2\pi)^n$, where $n = \dim_\C(M)$. Note
that for a quantum line bundle, this volume can be considered as
coming from the geometry of either $M$ or $L$.

In the rest of this section we assume that $M$ is compact, and so that
$\Omega$ has finite total volume. On each space $\Gamma(M,L^m)$ we
define the Hermitian inner product
$$
\langle\zeta,\xi\rangle = \int_M h^{(m)}(\zeta,\xi)\Omega, 
$$
where $\zeta,\xi \in \Gamma(M,L^m)$. The $L_2$-completion of the
latter Hermitian space is denoted by $L_2(M,L^m)$. Since $M$ is
compact, the space of global holomorphic sections of $L^m$, denoted by
$\Gamma_{hol}(M,L^m)$, is finite dimensional and so closed in
$L_2(M,L^m)$. In particular, we have an orthogonal projection
$$
\Pi_m : L_2(M,L^m) \rightarrow \Gamma_{hol}(M,L^m),
$$
for every positive integer $m$.

\section{Quantization and Bergman spaces on
  $\PC$}\label{sec:quantum-PC} 
In this section we recollect some known properties and facts of the
projective space $\PC$ that provide its quantization and the Bergman
spaces on it. For further details on the less elementary facts we
refer to \cite{Griffiths-Harris} and \cite{Schli}.

Recall that the tautological or universal line bundle of $\PC$ is
given by
$$
T = \{ ([w],z) \in \PC\times\C^{n+1} : z \in \C w\},
$$
and assigns to every point in $\PC$ the line in $\C^{n+1}$ that such
point represents. It is well known that $T$ is a holomorphic line
bundle. Furthermore, $T$ has a natural Hermitian metric $h_0$
inherited from the usual Hermitian inner product on $\C^{n+1}$. Let us
denote by $H = T^*$ the dual line bundle with the corresponding
induced metric $h$ dual to the metric $h_0$ on $T$. The line bundle
$H$ is called the hyperplane line bundle. The following well known
result provides a quantization for $\PC$ as described in
Section~\ref{sec:quantum}.

\begin{proposition}
  The Hermitian line bundle $(H,h)$ is a quantum line bundle over
  $\PC$. 
\end{proposition}

For every $m\in \mathbb{Z}_+$ and with respect to the coordinates
given by $\varphi_0$, the weighted measure on $\PC$ with weight $m$ is
given by
$$
d\nu_m(z) = \frac{(n+m)!}{m!}\frac{\Omega(z)}{(1 + |z_1|^2 + \dots +
  |z_n|^2)^m},
$$
which has the following explicit expression
$$
d\nu_m(z) = \frac{(n+m)!}{\pi^n m!} \frac{dV(z)}{(1 + |z_1|^2 +
    \dots + |z_n|^2)^{n+m+1}}
$$
where, as before, $dV(z) = dx_1\wedge dy_1 \wedge \dots \wedge
dx_n\wedge dy_n$ is the Lebesgue measure on $\C^n$. For the sake of
simplicity, we will use the same symbol $d\nu_m$ to denote the
weighted measures for both $\PC$ and $\C^n$. Note that that $d\nu_m$
is a probability measure. Following the remarks of
Section~\ref{sec:quantum}, the Hilbert space $L_2(\PC, H^m )$ denotes
the $L_2$-completion of $\Gamma(\PC, H^m )$ with respect to the inner
product defined using the measure $d\nu_m$.

The line bundle $H$ can be trivialized over each subset $U_j\subset
\PC$ ($U_j$ as in Section~\ref{sec:geom-PC}) so that the corresponding
set of transition functions for $H^m$ are given as follows
\begin{align*}
  g^m_{kj} : U_j\cap U_k &\rightarrow \C^* \\
  [w] &\mapsto \frac{w_k^m}{w_j^m}.
\end{align*}
In particular, there is a trivialization of $H^m$ over $U_0$. Then,
every section $\zeta$ of $H^m$ restricted to $U_0$ can be considered
as a map $\zeta|_{U_0} : U_0 \rightarrow \C$. Since $\varphi_0^{-1} :
\C^n \rightarrow U_0$ defines a biholomorphism, the composition
$\widehat{\zeta} = \zeta|_{U_0}\circ\varphi_0^{-1}$ maps $\C^n
\rightarrow \C$. Then, we have the following well known result.

\begin{proposition}\label{prop:isoPC_C}
  The map given by
  \begin{align*}
    \Phi_0 : L_2(\PC, H^m) &\rightarrow L_2(\C^n, \nu_m) \\
    \zeta &\mapsto \widehat{\zeta},
  \end{align*}
  is an isometry of Hilbert spaces.
\end{proposition}

The weighted Bergman space on $\PC$ with weight $m\in \Z_+$ is defined
by
\begin{align*}
  \mA^2_m(\PC) &= \{ \zeta \in L_2(\PC,H^m) : \zeta \mbox{ is
    holomorphic}
  \}\\
  &= \Gamma_{hol}(\PC,H^m).
\end{align*}

As remarked before, since $\PC$ is compact, the space $\mA^2_m(\PC)$
is finite dimensional. Furthermore, the next well known result
provides an explicit description of these Bergman spaces.

\begin{proposition}\label{prop:sections_polys}
  For every $m \in \Z_+$, the Bergman space $\mA^2_m(\PC)$ satisfies
  the following properties.
  \begin{enumerate}
  \item With respect to the homogeneous coordinates of $\PC$, the
    Bergman space $\mA^2_m(\PC)$ can be identified with the space
    $P^{(m)}(\C^{n+1})$ of homogeneous polynomials of degree $m$ over
    $\C^{n+1}$.
  \item For $\Phi_0$ the isometry from Proposition~\ref{prop:isoPC_C},
    we have
    $$
    \Phi_0(\mA^2_m(\PC)) = P_m(\C^n),
    $$
    where $P_m(\C^n)$ denotes the space of polynomials on $\C^n$ of
    degree $\leq m$.
  \end{enumerate}
\end{proposition}

Propositions~\ref{prop:isoPC_C} and \ref{prop:sections_polys} allow us
to reduce our computations on the spaces $L_2(\PC,H^m)$ and
$\mA^2_m(\PC)$ to the spaces $L_2(\C^n,\nu_m)$ and $P_m(\C^n)$,
respectively. In what follows and when needed, we will use such
reductions by applying the corresponding identifications without
further mention.

The Bergman space $\mA^2_m(\PC)$ identified with $P_m(\C^n)$ has the
natural monomial basis which we will use for our computations. Such
monomials are denoted by $z^p=z_1^{p_1}\dots z_n^{p_n}$ where
$p=(p_1,\ldots, p_n)$ and $|p| = p_1 + \dots + p_n \leq m$. We denote
the enumerating set by $J_n(m) =\{ p\in\Z^n_+ : |p| \leq m \}$. A
direct computation shows that 
\begin{align*} 
  \langle z^p,z^p \rangle_m &=\frac{(n+m)!}{\pi^n m!} \int_{\C^n}
  \frac{z^p\overline{z}^p dV(z)}{(1+ |z_1|^2 + \cdots
    +|z_n|^2)^{n+m+1}} \\
  &=\frac{(n+m)!}{m!} \int_{\R^n_+} \frac{t_1^{p_1}\cdots
    t_n^{p_n} dt_1\cdots dr_n}{(1+t_1+\cdots +t_n)^{n+m+1}} \\
  &= \frac{p! (m-|p|)!}{m!}
\end{align*} 
for every $p,q \in J_n(M)$, where $p!=p_1!\cdots p_n!$ and $p \in
J_n(m)$. Also, it is easy to check that $\langle z^p,z^q \rangle_m =
0$ for all $p,q \in J_n(m)$ such that $p \not= q$. In particular, the
set
\begin{equation}\label{basis-Bergman} 
  \left\{\left(\frac{m!}{p!
        (m-|p|)!}\right)^{\frac{1}{2}}z^p: p\in J_n(m) \right\}
\end{equation}  
is an orthonormal basis of $\mA^2_m(\PC)$.

The following result provides the classical description of the
weighted Bergman projections of $L_2(\PC,H^m)$ onto $\mA^2_m(\PC)$. In
particular, these projections are precisely the orthogonal maps
$\Pi_m$ from Section~\ref{sec:quantum} for $\PC$.

\begin{proposition}\label{prop:Bm}
  Let $B_m : L_2(\PC,H^m) \rightarrow L_2(\PC,H^m)$ be the operator
  given by the expression
  $$
  B_m(\psi)(z)=\frac{(n+m)!}{\pi^n m!} \int_{\C^n}
  \frac{\psi(w) K(z,w)dV(w)}{(1+|w_1|^2+\cdots
    +|w_n|^2)^{n+m+1}},
  $$ 
  where $K(z,w) = (1+z_1\overline{w}_1\cdots
  +z_n\overline{w}_n)^m$. Then, $B_m$ satisfies the following
  properties
  \begin{enumerate}
  \item If $\psi \in L_2(\PC,H^m)$, then $B_m(\psi) \in \mA^2_m(\PC)$.
  \item $B_m(\psi)=\psi$ for every $\psi\in \mA^2_m(\PC)$.
  \end{enumerate}
  In particular, $B_m$ is the orthogonal projection $L_2(\PC,H^m) \rightarrow
  \mA^2_m(\PC)$. Also, $K(z,w)$ is the Bergman kernel for $L_2(\PC,H^m)$.
\end{proposition}

\section{Toeplitz operators with separately radial
  symbols}\label{sec:radial} 
We introduce a decomposition for the projective space $\PC$ which is
similar in spirit to the quasi-elliptic decomposition used for the
$n$-dimensional unit ball in \cite{QV-Ball1}.

Consider the polar coordinates $z_j=t_jr_j$ where $t_j\in \T$ and $r_j
\in \R_+$, for every $j=1,\ldots ,n$. This yields, for all points $z
\in \C^n$, an identification
$$
z=(z_1,\ldots,z_n)=(t_1r_1,\ldots,t_nr_n)=(t,r),
$$
where $t=(t_1,\ldots, t_n)\in \T^n$, $r=(r_1,\ldots, r_n)\in
\R_+^n$. In particular, we have $\C^n=\T^n\times
\R_+^n$, with the corresponding volume form
$$
dV(z)=\prod_{j=1}^{n} \frac{dt_j}{it_j}\prod_{j=1}^{n}r_j dr_j.
$$
Hence, for the measure $\nu_m$ on $\C^n$ introduced in
Section~\ref{sec:quantum-PC}, we obtain the decomposition
$$
L_2(\C^n, \nu_m)=L_2(\T^n)\otimes L_2(\R_+^n,\mu_m),  
$$
where
$$
L_2(\T^n)=\bigotimes_{j=1}^{n}L_2\left(\T,\frac{dt_j}{2\pi
    it_j}\right)
$$
and the measure $d\mu_m$ of $L_2(\R_+^n,\mu_m)$ is given by
$$
d\mu_m=\frac{(n+m)!}{m!}(1+r_1^2+\cdots +r_n^2)^{-n-m-1}
\prod_{j=1}^{n}r_j dr_j.
$$

We note that the Bergman space $\mA^2_m(\PC)$ is given, in the local
coordinates from $\varphi_0$, as the (closed) subspace of
$L_2(\C^n,\nu_m)$ which consists of all functions satisfying the
equations 
$$
\frac{\partial \varphi}{\partial \overline{z}_j}=\frac{1}{2}\left(
  \frac{\partial}{\partial x_j}+i\frac{\partial}{\partial y_j}
\right)\varphi=0, \quad j=1,\ldots,n,
$$
or, in polar coordinates,
$$
\frac{\partial\varphi}{\partial \overline{z}_j}=\frac{t_j}{2}\left(
  \frac{\partial}{\partial
    r_j}-\frac{t_j}{r_j}\frac{\partial}{\partial t_j}
\right)\varphi=0, \quad j=1,\ldots,n.
$$

Now consider the discrete Fourier transform $\F:L_2(\T)\rightarrow
l_2=l_2(\Z)$ defined by
$$
\F: f\mapsto \left\{c_j=\int_{S^1}f(f)t^{-j}\frac{dt}{2\pi i
    t}\right\}_{j\in \Z},
$$
so that, in particular, the operator $\F$ is unitary with inverse
given by
$$
\F^{-1}=\F^*: \{ c_j\}_{j\in \Z}\mapsto
f=\sum_{j\in\Z}c_j t^j
$$
Now, let us consider the operator $u : l_2\otimes
L_2((0,1),rdr)\rightarrow l_2\otimes L_2((0,1),rdr)$ given by the
composition
$$
u=(\F\otimes I) \frac{t}{2}\left( \frac{\partial}{\partial
    r}-\frac{t}{r}\frac{\partial}{\partial t} \right)(\F^{-1}\otimes
I). 
$$
Then, it is easy to check (see Subsection~4.1 of \cite{vasilev-IE})
that $u$ acts by
$$
\{ c_j(r)\}_{j\in \Z}\mapsto \left\{ \frac{1}{2}
  \left(\frac{\partial}{\partial r}-\frac{j}{r}\right) c_j(r)
\right\}_{j\in\Z} 
$$

Introduce the unitary operator
$$
U=\F_{(n)}\otimes I:L_2(\T^n)\otimes
L_2(\R_+^n,\nu_m)\rightarrow l_2(\Z^n)\otimes
L_2(\R_+^n,\nu_m)
$$
where $\F_{(n)}=\F\otimes\cdots\otimes\F$ ($n$ times). Then, the image
$\mA^2_m=U(\mA^2_m(\PC))$ under $U$ of the Bergman space is the closed
subspace of $l_2(\Z^n)\otimes L_2(\R_+^n,\nu_m)$ which
consist of all sequences $\{c_p(r)\}_{p\in \Z^n}$,
$r=(r_1,\dots,r_n)\in \R^n_+$, satisfying the equations
$$
\frac{1}{2}  \left(\frac{\partial}{\partial r_j}-\frac{p_j}{r_j}\right)
c_{p}(r)=0,
$$
for all $p_j \in \Z$ and $j = 1,\dots,n$. The general solution of this
system of equations has the form
$$
c_p(r)=\alpha^m_p c_p r^p, 
$$
for all $p \in \Z^n$, where $c_p\in \C$, $r^p=r_1^{p_1}\cdots
r_n^{p_n}$ and $\alpha_p=\alpha_{(|p_1|,\ldots,|p_n|)}$ is given by 
\begin{eqnarray*} 
  \alpha^m_p
  &=& \left( \frac{ (n+m)!}{ m!}
    \int_{\R^n_+} \frac{t_1^{p_1}\cdots t_n^{p_n}t_1 dt_1\cdots
      t_n d_n}{(1+t_1+\cdots +t_n)^{n+m+1}} \right)^{-\frac{1}{2}} \\ 
  &=&
  \left(\frac{m!}{p! (m-|p|)!}\right)^{\frac{1}{2}}
\end{eqnarray*}

We have that every function $c_p(r)=\alpha^m_p c_p r^p$ has to be in
$L_2(\R_+^n,\nu_m)$, and this integrability condition implies $c_p=0$
for every $p\in \Z^n\setminus J_n(m)$.  Hence, $\mA^2_m\subset
l_2(\Z^n)\otimes L_2(\R_+^n,\nu_m)$ coincides with the space of all
sequences that satisfy
\begin{equation} 
  c_p(r)= \left\{
    \begin{array}{cc} 
      \left(\frac{m!}{p! (m-|p|)!}\right)^{\frac{1}{2}} c_p r^p & p\in
      J_n(m) \\ 
      0 & p\in \Z^n\setminus J_n(m)
    \end{array}\right.,
\end{equation}
and for all such sequences we have
$$
\|\{c_p(r) \}_{p\in J_n}\|_{l_2(\Z^n)\otimes
  L_2(\R_+^n,\nu_m)}
=\| \{c_p \}_{p\in J_n} \|_{l_2(\Z^n)}.
$$
In particular, we have a natural isometric identification $\mA^2_m =
l_2(J_n(m))$.  These constructions allow us to consider the isometric
embedding
$$
R_0:l_2(J_n(m))\rightarrow l_2(\Z^n)\otimes
L_2(\R_+^n,\nu_m)
$$
defined by
$$
R_0: \{c_p \}_{p\in J_n(m)}
\mapsto 
c_p(r)= \left\{
  \begin{array}{cc} 
    \left(\frac{m!}{p! (m-|p|)!}\right)^{\frac{1}{2}} c_p r^p & 
    p\in J_n(m) \\ 
    0 & p\in \Z^n/J_n(m)
  \end{array}\right..
$$
For which the adjoint operator $R_0^*:l_2(\Z^n)\otimes
L_2(\R^n_+,\nu_m)\rightarrow l_2(J_n(m))$ is given by
$$
R_0^*:\{f_p(r)\}_{p\in \Z^n}\mapsto \left\{ \left(\frac{m!}{p!
      (m-|p|)!}\right)^{\frac{1}{2}}\int_{\R^n_+}r^p
  f_p(r)d\nu_m(r) \right\}_{p\in J_n(m)}.
$$ 
It is easily seen that
\begin{eqnarray*} 
  R_0^*R_0=I&:& l_2(J_n(m))\rightarrow l_2(J_n(m))\\
  R_0R_0^*=P_1&:&l_2(\Z^n)\otimes
  L_2(\R^n_+,\nu_m)\rightarrow l_2(J_n(m))
\end{eqnarray*}
where $P_1$ is the ortogonal projection of $l_2(\Z^n)\otimes
L_2(\R^n_+,\nu_m)$ onto $l_2(J_n(M)) = \mA^2_m$. Summarizing the above
we have the next result.

\begin{theorem}\label{theo:RR*}
  The operator $R=R_0^*U$ maps $L_2(\mathbb{P}^n(\C),\nu_m)$ onto
  $\mA^2_m = l_2(J_n(m))$, and its restriction
  $$
  R\arrowvert_{\mA^2_m(\mathbb{P}^n(\C))} :{\mA^2_m(\PC)}\rightarrow
  l_2(J_n(m)) 
  $$
  is an isometric isomorphism. The adjoint operator
  $$
  R^*=U^* R_0:l_2(J_n(m)) \rightarrow \mA^2_m(\PC)\subset L_2(\PC,\nu_m) 
  $$
  is an isometric isomorphism of $l_2(J_n(m))=\mA^2_m$ onto the subspace
  $\mA^2_m(\PC)$. Furthermore
  \begin{align*} 
    R R^*=I : l_2(J_n(m)) &\rightarrow l_2(J_n(m))\\
    R^*R=B_m : L_2(\PC,\nu_m) &\rightarrow  \mA^2_m(\PC)
  \end{align*} 
  where $B_m$ is the Bergman projection of $L_2(\PC,\nu_m)$ onto
  $\mA^2_m(\PC)$
\end{theorem}

We note that an explicit calculation yields
\begin{align*} 
  R^*=U^*R_0 : \{c_p\}_{p\in J_n(m)} &\mapsto 
  U^*\left(\left\{\left(\frac{m!}{p! (m-|p|)!}\right)^{\frac{1}{2}}
      c_p r^p \right\}_{p\in J_n(m)}\right) \\  
  &=\sum_{p\in J_n(m)} \left(\frac{m!}{p!
      (m-|p|)!}\right)^{\frac{1}{2}} c_p (rt)^p \\ 
  &=\sum_{p\in J_n(m)}\left(\frac{m!}{p!
      (m-|p|)!}\right)^{\frac{1}{2}} c_p z^p,
\end{align*}
which implies the following result.

\begin{corollary} 
  With the notation of Theorem~\ref{theo:RR*}, the isometric
  isomorphism $R^*:l_2(J_n(m)) \rightarrow \mA^2_m(\PC)\subset
  L_2(\PC,\nu_m)$ is given by
  \begin{equation}\label{R-adjoin} 
    R^*:\{ c_p\}_{p\in J_n(m)}\mapsto \sum_{p\in J_n(m)}
    \left(\frac{m!}{p! (m-|p|)!}\right)^{\frac{1}{2}} c_p z^p.
  \end{equation}
\end{corollary}

A similar direct computation yields the following result.

\begin{corollary} 
  With the notation of Theorem~\ref{theo:RR*}, the isometric
  isomorphism $R\arrowvert_{\mA^2_m(\PC)} :{\mA^2_m(\PC)}\rightarrow
  l_2(J_n(m))$ is given by
  \begin{equation}\label{R-isometric} 
    R:\psi\mapsto \left\{\left(\frac{m!}{p!
          (m-|p|)!}\right)^{\frac{1}{2}} \int_{\C^n}
      \psi(z)\overline{z}^p d\nu_m(z) \right\}_{p\in J_n(m)} .
  \end{equation}
\end{corollary}

We now introduce a special family of symbols on $\C^n$.

\begin{definition}
  We will call a function $a(z)$, $z\in \C^n$ separately radial if
  $a(z)=a(r)$, i.e. $a$ depends only on the radial components of
  $z=(t,r)$.
\end{definition}

The separately radial symbols give rise to a $C^*$-algebra of Toeplitz
operators that can be turned simultaneously into multiplication
operators. 

\begin{theorem} \label{theo:comm-C*-algebra} Let $a$ be a bounded
  measure separately radial function. Then the Toeplitz operator $T_a$
  acting on $\mA^2_m(\PC)$ is unitary equivalent to the multiplication
  operator $\gamma_{a,m}I=RT_aR^*$ acting on $l_2(J_n(m))$, where $R$
  and $R^*$ are given by $\mathrm{(\ref{R-isometric})}$ and
  $\mathrm{(\ref{R-adjoin})}$ respectively. The sequence
  $\gamma_{a,m}=\{\gamma_{a,m}(p)\}_{p\in J_n(m)}$ is explicitely
  given by
  \begin{equation}
    \gamma_{a,m}(p)=\frac{2^n m!}{p!
      (m-|p|)!}\int_{\R^n_+} \frac{a(r_1,\dots,r_n)
      r_1^{2p_1+1}\cdots r_n^{2p_n+1} dr_1\cdots dr_n}{(1+r_1^2+\cdots
      +r_n^2)^{n+m+1}}.
  \end{equation}
\end{theorem} 
\begin{proof}
  Using the previous results, the operator $T_a$ is unitary equivalent
  to the operator
  \begin{align*} 
    R T_a R^* &= R B_m a B_m R^*=R(R^* R)a(R^*R )R^* \\
    &= (R R^*)Ra R^* (R R^*)=RaR^* \\ 
    &= R^*_0 U a U^* R_0 \\ 
    &= R^*_0 (\F_{(n)}\otimes I) a (\F^{-1}_{(n)}\otimes I) R_0 \\
    &=R^*_0 a R_0.
  \end{align*}
  And the latter operator is computed as follows
  \begin{align*} 
    &R^*_0 a(r) R_0\{c_p\}_{p\in J_n(m)} \\
    &= R^*_0 \left\{\left(\frac{m!}{p! (m-|p|)!}\right)^{\frac{1}{2}}
      a(r)c_pr^p\right\}_{p\in J_n(m)} \\
    &= \left\{\left(\frac{m!}{p!(m-|p|)!}\right)^{\frac{1}{2}}
      \int_{\R^n_+} \frac{a(r)c_pr^{2p}r dr}{(1+r_1^2+\cdots
        +r_n^2)^{n+m+1}} \right\}_{p\in J_n(m)} \\
    &=\{\gamma_{a,m} (p)\cdot c_p \}_{p\in J_n(m)}.
  \end{align*}
\end{proof}
  
Hence we can diagonalize the corresponding Toeplitz operators.

\begin{corollary} 
  The Toeplitz operator $T_a$ with bounded measurable separately
  radial symbol $a(r)$ is diagonal with respect to the orthonormal
  base given by $\mathrm{(}\ref{basis-Bergman}\mathrm{)}$. More
  precisely, we have
  $$
  T_a\left( \left(\frac{m!}{p! (m-|p|)!}\right)^{\frac{1}{2}}z^p
  \right)=\gamma_{a,m}(p)\left(\frac{m!}{p!
      (m-|p|)!}\right)^{\frac{1}{2}}z^p
  $$
  for all $p \in J_n(m)$.
\end{corollary}

\section{Abelian groups and Lagrangian frames on
  $\PC$}\label{sec:abel-lagr}
In this section, we will discuss the geometric features of separately
radial symbols. Such description will be provided in the context of
both actions of Lie groups and foliated spaces.

Let us denote with $\wA(n)$ the subgroup of diagonal matrices in the
Lie group $\SU(n+1)$, which is obviously a connected Abelian Lie
subgroup of dimension $n$. Correspondingly, we denote $\A(n) =
\lambda(\wA(n))$, which is a connected Abelian Lie subgroup of
$\PSU(n+1)$ of dimension $n$. Recall from Section~\ref{sec:geom-PC}
that $\lambda$ denotes the natural projection map $\SU(n+1)
\rightarrow \PSU(n+1)$. The group $\A(n)$ describes our separately
radial symbols by use of the next result, whose proof follows readly
from the definitions involved.

\begin{lemma}\label{lem:radialgroup}
  Let $\varphi_0 : U_0 \rightarrow \C^n$ be the coordinate chart of $\PC$
  defined in Section~\ref{sec:geom-PC}. Then, $U_0$ is
  $\A(n)$-invariant and induces on $\C^n$ the $\T^n$-action given by
  \begin{eqnarray*}
    \T^n\times\C^n &\rightarrow& \C^n \\
    (t,z) &\mapsto& tz = (t_1 z_1, \dots, t_n z_n). 
  \end{eqnarray*}
  More precisely, if we let $\rho : \T^n \rightarrow \wA(n)$ be the
  isomorphism defined by
  $$
  \rho(t) =
    \begin{pmatrix}
      t_1 & 0 & \dots & 0 & 0 \\
      0 & t_2 & \dots & 0 & 0 \\
      \vdots & \vdots & \ddots & \vdots & \vdots \\
      0 & 0 & \dots & t_n & 0 \\
      0 & 0 & \dots & 0 & \overline{t}_1\dots \overline{t}_n
    \end{pmatrix}
  $$
  then, we have
  $$
  t\varphi_0([w]) = \varphi_0([\rho(t)w]),
  $$
  for every $t \in \T^n$ and $[w] \in U_0$. In particular, a symbol $a
  : \C^n \rightarrow \C$ is separately radial if and only if
  $a\circ\varphi_0$ is $\A(n)$-invariant as a function on $U_0 \subset
  \PC$.
\end{lemma}

Note that a change of coordinates by a biholomorphism of $\PC$ can be
used to obtain from our separately radial symbols an equivalent family
of $C^*$-algebras of Toeplitz operators. Correspondingly, conjugation
by any such biholomorphism provides an Abelian group equivalent to
$\A(n)$. In our search for commutative $C^*$-algebras of Toeplitz
operators this suggests to study the relation of $\A(n)$ with other
possible Abelian groups, particularly those obtained by
conjugations. The following result provides a complete solution to
such problem within the group $\PSU(n+1)$ that defines the connected
component $\Iso_0(\PC)$ of the group of isometries of $\PC$.

\begin{theorem}\label{theo:MASGPC}
  The Lie subgroup $\A(n)$ of $\PSU(n+1)$ satisfies the following
  properties.
  \begin{enumerate}
  \item $\A(n)$ is isomorphic to $\T^n$.
  \item $\A(n)$ is a maximal Abelian subgroup
    $\mathrm{(MASG)}$ of $\PSU(n+1)$; i.e.~if $H$ is a
    connected Abelian proper subgroup of $\PSU(n+1)$ that contains
    $\A(n)$, then $H = \A(n)$.
  \item If $H$ is a connected Abelian Lie subgroup of $\PSU(n+1)$,
    then there exist $g \in \PSU(n+1)$ such that $gHg^{-1} \subset
    \A(n)$.
  \end{enumerate}
\end{theorem}
\begin{proof}
  First note that an isomorphism between $\A(n)$ and $\T^n$ was
  already provided in Lemma~\ref{lem:radialgroup}.

  The rest of the statement follows from the fact that $\PSU(n+1)$ is
  a compact connected group of matrices. We explain how to prove the
  remaining claims from the results found in \cite{Knapp}.

  By the well known correspondence between Lie subalgebras and Lie
  subgroups, we conclude that the MASG's of $\PSU(n+1)$ are precisely
  those whose Lie algebra are maximal Abelian subalgebras of the Lie
  algebra of $\PSU(n+1)$. Hence, it follows from Proposition~4.30 of
  \cite{Knapp} that the MASG are precisely the maximal tori in
  $\PSU(n+1)$. We recall that a torus is a group isomorphic to a
  product of circles, and that the maximality here refers to the order
  by inclusion. Hence, Example~1 in page~252 of \cite{Knapp} implies
  that $\A(n)$ is a MASG of $\PSU(n+1)$, which yields (2).

  Finally, the last assertion of the statement follows from
  Corollary~4.35 of \cite{Knapp} and our remark above on the fact that
  the MASG's in $\PSU(n+1)$ are precisely the maximal tori.
\end{proof}

As a consequence, we obtain the following uniqueness result for the
isometry group $\Iso_0(\PC)$. This follows from
Theorem~\ref{theo:MASGPC} and Proposition~\ref{prop-Iso-PSU}.

\begin{theorem}
  A subgroup $H$ of $\Iso_0(\PC)$ is a MASG if and only if there exist
  $\varphi \in \Iso_0(\PC)$ such that $\varphi H \varphi^{-1} = \A(n)$.
\end{theorem}

To continue the description of our separately radial symbols on $\PC$
we will use the notion of a foliation and its geometric properties.

We recall the definition of a foliation as found in \cite{QV-Ball2}.
On a smooth manifold $M$ a codimension $q$ foliated chart
is a pair $(\varphi, U)$ given by an open subset $U$ of $M$
and a smooth submersion $\varphi : U \rightarrow V$, where $V$ is an
open subset of $\R^q$. For a foliated chart $(\varphi, U)$ the
connected components of the fibers of $\varphi$ are called the plaques
of the foliated chart. Two codimension $q$ foliated charts
$(\varphi_1,U_1)$ and $(\varphi_2,U_2)$ are called compatible if there
exists a diffeomorphism $\psi_{12} : \varphi_1(U_1\cap U_2)
\rightarrow \varphi_2(U_1\cap U_2)$ such that the following diagram
commutes
\begin{equation}\label{charts-comp}
  \xymatrix{
    &   U_1\cap U_2 \ar[dl]_{\varphi_1} \ar[dr]^{\varphi_2}  & \\
    \varphi_1(U_1\cap U_2) \ar[rr]^{\psi_{12}}    &     &   \varphi_2(U_1\cap
    U_2).
  }
\end{equation}

A foliated atlas on a manifold $M$ is a collection
$\{(\varphi_\alpha,U_\alpha)\}_\alpha$ of foliated charts that are
mutually compatible and such that $M = \bigcup_\alpha
U_\alpha$.

The compatibility of two foliated charts $(\varphi_1,U_1)$ and
$(\varphi_2,U_2)$ is defined so that it ensures that, when restricted
to $U_1\cap U_2$, both submersions $\varphi_1$ and $\varphi_2$ have
the same plaques. This implies that the following is an equivalence
relation in $M$.
\begin{eqnarray*}
  x\sim y &\iff& \mbox{there is a sequence of plaques }
  (P_k)_{k=0}^l \mbox{ for foliated charts } \\
  && (\varphi_k, U_k)_{k=0}^l,  \mbox{respectively, of the foliated
    atlas, such that } \\
  && x\in P_0,\ y\in P_l, \mbox{ and }
  P_{k-1}\cap P_k\neq\phi \mbox{ for every }
  k=1,\dots,l
\end{eqnarray*}
The equivalence classes are submanifolds of $M$ of
dimension $\dim(M)-q$.

\begin{definition}\label{def-foliation}
  A foliation $\F$ on a manifold $M$ is a partition of
  $M$ that is given by the family of equivalence classes of
  the relation of a foliated atlas. The classes are called the leaves
  of the foliation.
\end{definition}

If $\F$ is a foliation, we denote with $T\F$ the space of tangent
vectors to the leaves of $\F$. The space $T\F$ is called the tangent
bundle of the foliation. We note that $T\F$ is a vector subbundle of
the tangent bundle to the ambient manifold. If $E$ is a vector
subbundle of the tangent bundle of the ambient manifold, then we will
say that $E$ is integrable if it is the tangent bundle of a
foliation. We observe that not every vector bundle is
integrable. Also, a vector bundle $E$ is integrable if and only if
through every point in the ambient manifold there is a submanifold $N$
such that $TN_p = E_p$ for every $p \in N$.

For the geometric description of our separately radial symbols we need to
consider Lagrangian foliations, i.e.~those for which the leaves are
Lagrangian submanifolds. We recall that a submanifold $N$ of a
symplectic manifold $M$ is called Lagrangian if the tangent space
$T_pN$ is a Lagrangian subspace of $T_pM$ for every $p \in N$.

We will also consider further properties for submanifolds which we
collect in the next definition. We refer to \cite{KNII} for more
details.

\begin{definition}
  Let $M$ be a Riemannian manifold with Levi-Civita connection
  $\nabla$, and $N$ a submanifold of $M$.
  \begin{enumerate}
  \item $N$ is called complete if every geodesic in $N$ can be
    defined for every value of $\R$ so that it still lies in $N$.
  \item $N$ is called flat if it has vanishing sectional curvature for
    the metric inherited from $M$.
  \item $N$ is called parallel if $\nabla \Pi = 0$, where $\Pi$ is the
    second fundamental form of $N$ with respect to $M$.
  \item $N$ is called totally geodesic if every geodesic in $N$ is a
    geodesic in $M$ as well.
  \end{enumerate}
\end{definition}

Given these definitions, the following provides the main object that
will be used to geometrically describe the separately radial symbols on $\PC$.

\begin{definition} 
  A Lagrangian frame on an open connected subset $D$ of $\PC$ is a
  pair of foliations $(\F_1, \F_2)$ that satisfy the following
  properties:
  \begin{enumerate}
  \item The leaves of $\F_1$ are complete flat parallel Lagrangian
    submanifolds.
  \item The leaves of $\F_2$ are totally geodesic Lagrangian
    submanifolds.
  \item At their intersection, every leaf of $\F_1$ is perpendicular
    to every leaf of $\F_2$.
  \end{enumerate}
\end{definition}

It has been shown that Lagrangian frames on the $n$-dimensional
complex unit ball provide the natural geometric setup to study symbols
generating commutative $C^*$-algebras of Toeplitz operators (see
\cite{QV-Ball1} and \cite{QV-Ball2}). In our case of the separately
radial symbols for $\PC$ we will show that a similar phenomenon takes
place. We show that the geometry of the level sets of separately
radial symbols yield a Lagrangian frame.

First, we present the full description of the complete flat parallel
submanifolds of $\PC$ as found in \cite{Naitoh-Takeuchi}.

Let us denote with $S^1(r)$ the circle in $\C$ with radius $r > 0$ and
centered at the origin. For every $r \in S^n\cap\R^{n+1}_+$, the torus
$$
\widehat{M}(r) = S^1(r_1)\times \dots \times S^1(r_{n+1}) \subset
\C^{n+1}
$$ 
is clearly contained in $S^{2n+1}$. As above, we let $\pi : S^{2n+1}
\rightarrow \PC$ denote the Hopf fibration of $\PC$. We will also
denote $M(r) = \pi(\widehat{M}(r))$ for every $r \in
S^n\cap\R^{n+1}_+$.  The following result is a consequence of
Theorems~2.1 and 3.1 from \cite{Naitoh-Takeuchi}.

\begin{theorem} \label{thm-flat-parallel}
  For every $r \in S^n\cap\R^{n+1}_+$, the set $M(r)$ is a
  connected complete flat parallel Lagrangian submanifold of
  $\PC$. Furthermore, if $M$ is any connected complete flat parallel
  Lagrangian submanifold of $\PC$, then there exists $\varphi \in
  \Iso_0(\PC)$ such that $\varphi(M) = M(r)$ for some $r \in
  S^n\cap\R^{n+1}_+$.
\end{theorem}

Let us denote
$$
\PC_0 = \{ [z_0, \dots, z_n] \in \PC : z_j \not= 0 \mbox{ for every }
j = 0, \dots, n \}.
$$
Clearly, $\PC_0$ is a connected open subset of $\PC$ which is conull
and dense as well. Also, the action of $\A(n)$ clearly preserves
$\PC_0$ and restricted to this set it is free. In other words, if for
some $p \in \PC_0$ and $g \in \A(n)$ we have $gp = p$, then $g = e$
(the identity element).

The following result provides a characterization of the flat parallel
Lagrangian submanifolds in terms of the MASG $\A(n)$ of $\Iso_0(\PC)$.

\begin{theorem} 
  \label{thm-A-orbits} If $p \in \PC_0$, then the orbit $\A(n) p$ is a
  connected complete flat parallel Lagrangian submanifold of
  $\PC$. Conversely, for every connected complete flat parallel
  Lagrangian submanifold $M$ of $\PC$ there exists $p \in \PC_0$ such
  that $M = \A(n) p$.
\end{theorem}
\begin{proof}
  Let $p \in \PC_0$ be given and let $r_j = |z_j|$, for $j = 0, \dots,
  n$, where $p = [z_0, \dots, z_n]$. Without loss of generality, we
  can assume $z \in S^{2n+1}$ so that in particular $r = (r_0, \dots,
  r_n) \in S^n \cap \R^{n+1}_+$.
  
  By the definition of $\A(n)$ and the $\PSU(n+1)$-action on $\PC$, it
  follows that
  $$
  \A(n) p = \{ [t_0 r_0, \dots, t_n r_n] : t_j \in S^1, j = 0, \dots,
  n\}.
  $$
  Hence, we have $\A(n) p = M(r)$ and so
  Theorem~\ref{thm-flat-parallel} implies that $\A(n) p$ is a
  connected complete flat parallel Lagrangian submanifold of $\PC$. 

  Let us now assume that $M$ is a connected complete flat parallel
  Lagrangian submanifold of $\PC$. By Theorem~\ref{thm-flat-parallel}
  there exist $\varphi \in \Iso_0(\PC)$ and $r \in S^n\cap\R^{n+1}_+$
  such that $\varphi(M) = M(r)$. But the above computation shows 
  that $M(r) = \A(n) p$ for $p = [r_0, \dots, r_n] \in \PC$, thus
  completing the proof.
\end{proof}

Let us denote with $\OO$ the set of $\A(n)$-orbits in $\PC_0$. The
next result shows that $\OO$ allows to obtain a Lagrangian frame of
$\PC$ defined on $\PC_0$.

\begin{theorem}\label{thm-A-Lag-frame}
  The collection $\OO$ of $\A(n)$-orbits on $\PC_0$ satisfies the
  following properties:
  \begin{enumerate}
  \item $\OO$ is a foliation whose leaves are complete flat parallel
    Lagrangian submanifolds of $\PC$.
  \item The orthogonal complement $T\OO^\perp$ of $T\OO$ is
    integrable and the leaves of the associated foliation $\OO^\perp$
    are totally geodesic Lagrangian submanifolds of $\PC$.
  \end{enumerate}
  In particular, the pair $(\OO, \OO^\perp)$ defines a Lagrangian
  frame on the open subset $\PC_0$ of $\PC$.
\end{theorem}
\begin{proof}
  We already noted that the $\A(n)$-action is free and preserves the
  Riemannian metric of $\PC$. Hence, $\OO$ defines a foliation by
  Proposition~6.7 from \cite{QV-Ball2}. Furthermore, by
  Theorem~\ref{thm-A-orbits} the leaves of $\OO$ are complete flat
  parallel Lagrangian submanifolds. This establishes (1).

  Next we prove that $T\OO^\perp$ is integrable. Since the leaves of
  $\OO$ are Lagrangian, so are the fibres of the vector bundle
  $T\OO^\perp$. In particular, we have $T\OO^\perp = i T\OO$. For a
  given $p \in \PC_0$ let us consider the space
  $$
  N(p) = \{ [r_0 z_0, \dots, r_n z_n] : r_j \in \R_+ \mbox{ for every
  } j = 0, \dots, n \}
  $$
  where $p = [z_0, \dots, z_n]$. It is clear that $N(p)$ is a
  submanifold of $\PC$ and that $T_q N(p) = i T_q\OO$ for every $q \in
  N$. This proves the integrability of $T\OO^\perp$.

  Finally, that the leaves of $\OO^\perp$ are totally geodesic is a
  direct consequence of Proposition~6.9 from \cite{QV-Ball2}.
\end{proof}

The following result shows that every Lagrangian frame is, up to an
isometry, the Lagrangian frame $(\OO, \OO^\perp)$ defined above, and
thus it is given by our separately radial symbols.

\begin{theorem}\label{thm-Lag-frame-classification} 
  If $(\F_1, \F_2)$ is a Lagrangian frame defined in an open connected
  subset $U$ of $\PC$, then there exist $\varphi \in \Iso_0(\PC)$ of
  $\PC$ such that:
  \begin{enumerate}
    \item $\varphi(U) \subset \PC_0$.
    \item Every leaf of $\varphi(\F_1)$ is a leaf of $\OO$.
    \item Every leaf of $\varphi(\F_2)$ is an open subset of a leaf of
      $\OO^\perp$.
  \end{enumerate}
\end{theorem}
\begin{proof}
  Let $L$ be a leaf of $\F_1$. Hence, $L$ is a connected complete flat
  parallel Lagrangian submanifold of $\PC$. By
  Theorem~\ref{thm-A-orbits}, there exist $\varphi \in \Iso_0(\PC)$
  and $p \in \PC_0$ such that $\varphi(L) = \A(n) p$. Let us consider
  the image under the exponential map $\exp$ of the normal bundle
  $N$ to $\A(n)p$. Then, for every $q \in \A(n)p$ the set $\exp(N_q)$
  is the largest totally geodesic submanifold of $\PC$ perpendicular
  to $\A(n)p$ at $q$. Since $(\F_1,\F_2)$ is a Lagrangian frame, it
  follows that for $L'$ the leaf of $\F_2$ through $\varphi^{-1}(q)$
  we have
  $$
  \varphi(L') \subset \exp(N_q).
  $$
  Since $(\OO, \OO^\perp)$ is a Lagrangian frame as well, we conclude
  that $\exp(N_q)$ is a leaf of $\OO^\perp$. This proves (3). But
  then (2) follows since the leaves of $\F_1$ (resp. $\OO$) are the
  integral submanifolds of the orthogonal complement of $T\F_1$
  (resp. $T\OO$). 
  
  Finally (1) follows from the fact that $U$ is the union of the
  leaves of $\F_1$.
\end{proof}

We use the previous results to prove that every family of symbols
associated to a Lagrangian frame is, up to a biholomorphism, a subset
of our separately radial symbols.

\begin{theorem}\label{theo:lagrframes-radialsymbols}
  For a subspace $\mathcal{A}$ of $C^\infty(\PC)$ the following
  conditions are equivalent.
  \begin{enumerate}
  \item There is a Lagrangian frame $(\F_1,\F_2)$ defined in a
    connected open conull subset $U$ of $\PC$ such that if $a \in
    \mathcal{A}$, then every level set of $a$ is saturated with
    respect to the foliation $\F_1$, i.e., every such level set is a
    union of leaves of $\F_1$.
  \item There exist $\varphi \in \Iso_0(\PC)$ such that $\mathcal{A}
    \subset \varphi^*(\mathcal{A}_{\A(n)}) = \{a\circ\varphi : a \in
    \mathcal{A}_{\A(n)}\}$, where $\mathcal{A}_{\A(n)}$ is the
    subspace of $C^\infty(\PC)$ consisting of $\A(n)$-invariant
    functions.
  \end{enumerate}
\end{theorem}

\begin{proof}
  That (2) implies (1) is the content of Theorem~\ref{thm-A-Lag-frame}
  together with (the obvious) invariance of Lagrangian frames with
  respect to elements in $\Iso_0(\PC)$.

  To prove that (1) implies (2) we use
  Theorem~\ref{thm-Lag-frame-classification}. From this result it
  follows that there exist $\varphi \in \Iso_0(\PC)$ such that the
  Lagrangian frame $(\F_1,\F_2)$ is mapped under $\varphi^{-1}$ to a
  restriction of $(\OO,\OO^\perp)$. Hence, for every $a \in
  \mathcal{A}$, the level subsets of $a \circ \varphi$ are saturated
  with respect to the leaves of the foliation $\OO$ on
  $\varphi^{-1}(U)$. This implies that, for every such $a$, the
  function $a \circ \varphi$ is $\A(n)$-invariant on
  $\varphi^{-1}(U)$.  Hence, the result follows by the density of $U$
  that comes from the fact that it is conull.
\end{proof}

\begin{corollary}
  Given any Lagrangian frame $\F=(\F_1,\F_2)$ on $\PC$, denote by
  $\mathcal{A}_{\F}$ the set of all functions in $C^{\infty}(\PC)$
  which are constant on the leaves of $\F_1$. Then the $C^*$-algebra
  $\mT_{\F}$ generated by the Toeplitz operators with symbols in
  $\mathcal{A}_{\F}$ is commutative on each weighted Bergman space
  $\mathcal{A}^2_m(\PC)$, $m \in \Z_+$. Furthermore, $\mT_{\F}$ is
  unitarily equivalent to $\mT_{(\OO, \OO^\perp)}$.
\end{corollary}

\end{document}